\documentclass[10.5pt,oneside,reqno]{amsart}
\usepackage{textcomp}
\usepackage{amsmath}
\usepackage{mathrsfs}
\usepackage{amsthm}
\usepackage[linktocpage]{hyperref}
\usepackage{amsfonts,graphics,amsthm,amsfonts,amscd,latexsym}
\usepackage{epsfig}
\usepackage{tikz}
\usepackage{tikz-cd}
\usepackage{wasysym}
\usepackage{flafter}
\usepackage{mathtools}
\usepackage{comment}
\usepackage{stmaryrd}
\usepackage{enumitem}
\usepackage{epigraph}

\hypersetup{
colorlinks=true,
linkcolor=red,
citecolor=green,
filecolor=blue,
urlcolor=blue
}
\usepackage{tikz}
\usetikzlibrary{graphs,positioning,arrows,shapes.misc,decorations.pathmorphing}

\tikzset{
>=stealth,
every picture/.style={thick},
graphs/every graph/.style={empty nodes},
}

\tikzstyle{vertex}=[
draw,
circle,
fill=black,
inner sep=1pt,
minimum width=5pt,
]
\usepackage[position=top]{subfig}
\usepackage[alphabetic]{amsrefs}
\usepackage{amssymb}
\usepackage{color}

\calclayout

\usetikzlibrary{decorations.pathmorphing}
\tikzstyle{printersafe}=[decoration={snake,amplitude=0pt}]

\newcommand{\pp}{\mathbb{P}}

\newcommand{\rr}{\mathbb{R}}
\newcommand{\cc}{\mathbb{C}}

\def\O#1.{\mathcal {O}_{#1}}
\def\pr #1.{\mathbb P^{#1}}
\def\af #1.{\mathbb A^{#1}}
\def\ses#1.#2.#3.{0\to #1\to #2\to #3 \to 0}
\def\xrar#1.{\xrightarrow{#1}}
\def\K#1.{K_{#1}}
\def\bA#1.{\mathbf{A}_{#1}}
\def\bM#1.{\mathbf{M}_{#1}}
\def\bL#1.{\mathbf{L}_{#1}}
\def\bB#1.{\mathbf{B}_{#1}}
\def\bK#1.{\mathbf{K}_{#1}}
\def\subs#1.{_{#1}}
\def\sups#1.{^{#1}}

\DeclareMathOperator{\Supp}{Supp}

\usepackage{tikz}
\usetikzlibrary{matrix,arrows,decorations.pathmorphing}

\newtheorem{theorem}{Theorem}[section]
\newtheorem{lemma}[theorem]{Lemma}

\newtheorem{corollary}[theorem]{Corollary}
\newtheorem{claim}[theorem]{Claim}

\theoremstyle{definition}

\newtheorem{example}[theorem]{Example}

\newtheorem{remark}[theorem]{Remark}

\theoremstyle{remark}

\numberwithin{equation}{section}

\usepackage[all]{xy}

\newcounter{rownumber}[figure]
\newcounter{rownumber-irr}[figure]
\newcounter{rownumber-p1}[figure]
\begin{document}

\title{Real del Pezzo surfaces without points}

\author[G.~Belousov]{Grigory Belousov}

\maketitle

\begin{abstract}
We consider a real del Pezzo surface without points. We prove that the same surface over complex numbers field $\cc$ has Picard number is at least two.
\end{abstract}

\section{Introduction}

A \emph{log del Pezzo surface} is a projective algebraic surface X with only
log terminal singularities and ample anti-canonical divisor $-K_X$.

Del Pezzo surfaces naturally appear in the log minimal model program
(see, e. g., {\cite{KMM}}). The most interesting class of del Pezzo surfaces
is the class of surfaces with Picard number 1.

Let $X$ be a real de Pezzo surface. Put $\tilde{X}$ is the same surface over complex numbers field $\cc$.
Note that there exists Galois involution $\tau$ that acts on $\tilde{X}$. We see that $\rho(X)\leq\rho(\tilde{X})$. In this paper we consider del Pezzo surface over real number field $\rr$. We consider del Pezzo surface without points over $\rr$.

The main result of this paper is the following.

\begin{theorem}
\label{main-theorem}
Let $X$ be a real del Pezzo surface that does not contain points. Put $\tilde{X}$ be the del Pezzo surface over $\cc$. Assume that $\tilde{X}$ has only log terminal singularities. Then $\rho(\tilde{X})>1$.
\end{theorem}

\begin{example}
\label{rem1}
Note that the surface $X\subset\pp^3$ defined by $$x_0^2+x_1^2+x_2^2+x_3^2=0$$ has no points over $\rr$. So, there exist a del Pezzo surfaces $X$ with no points over $\rr$ and with $\rho(\tilde{X})=2$.
\end{example}

\begin{remark}
\label{rem2}
Assume that $\tilde{X}$ has only du Val singularities and $X$ has no points. Then the degree of $X$ is even.
\end{remark}

The author is grateful to Professor I. A. Cheltsov for proposing this problem and for useful comments.

\section{Preliminary results}

We employ the following notation:

\begin{itemize}
\item
$(-n)$-curve is a smooth rational curve with self intersection
number $-n$.
\item
$K_{X}$: the canonical divisor on $X$.
\item
$\rho(X)$: the Picard number of $X$.
\end{itemize}

\begin{theorem}[see {\cite[Corollary 9.2]{KeM}}]
\label{BMY} Let $X$ be a rational surface with log terminal
singularities and $\rho(X)=1$. Then
\[
\sum_{P\in X}\frac{m_P-1}{m_P}\leq 3,
\]
where
$m_P$ is the order of the local fundamental group $\pi_1(U_P-\{P\})$
($U_P$ is a sufficiently small neighborhood of $P$).
\end{theorem}

\begin{theorem}[Hodge, see, e.g., {\cite{Har}}, Theorem 1.9, Remark 1.9.1, Ch. 5]
\label{Hodge} The intersection form on a surface $X$ has the signature $(1,\rho(X)-1)$, where $\rho(X)$ is the Picard number of $X$.
\end{theorem}

Let $X$ be a real de Pezzo surface. Put $\tilde{X}$ is the same surface over complex numbers field $\cc$.
Note that there exists Galois involution $\tau$ that acts on $\tilde{X}$.

Let $\pi\colon\bar{X}\rightarrow\tilde{X}$ be the minimal resolution, $D=\sum D_i$ be the exceptional divisor. Put $\pi^*(K_{\tilde{X}})=K_{\bar{X}}+D^\sharp$, where $D^\sharp=\sum\alpha_i D_i$, $0<\alpha_i<1$. Let $C$ be a curve on $\bar{X}$ such that $-C\cdot(K_{\bar{X}}+D^\sharp)$ attains the smallest positive value.

\begin{claim}
\label{DuVal}
Assume that $\tilde{X}$ has only du Val singular points. Then $\rho(\tilde{X})>1$.
\end{claim}

\begin{proof}
By \ref{rem2} we see that the degree of $X$ is even. On the other hand, $\deg\tilde{X}=10-\rho(\bar{X})=10-\rho(\tilde{X})-\# D$. Since $X$ has no points, we see that there no exist $\tau$-invariant components of $D$. So, $\# D$ is even. Hence, $\rho(\tilde{X})$ is even. Then $\rho(\tilde{X})>1$.
\end{proof}

\begin{lemma}[{\cite[Lemma 1.5]{Zh}}]
\label{Zhan0}
Assume that $\phi\colon\bar{X}\rightarrow\pp^1$ is a conic bundle. Then the following
assertions hold:
\begin{enumerate}
\item $\#\{\text{irreducible components of $D$ not in any fiber of $\phi$}\}=1+\sum\limits_F(\#\{\text{$(-1)$-curves in $F$}\}-1)$, where $F$ moves over all singular fibers of $\phi$.
\item If a singular fiber $F$ consists only of $(-1)$-curves and $(-2)$-curves then $F$ has
one of the following dual graphs:
$$
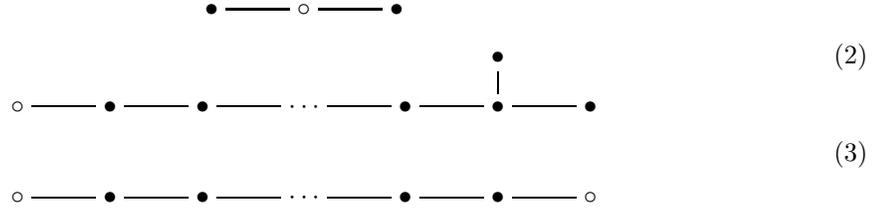
\xymatrix@R=0.8em{\\
\bullet\ar@{-}[r]\ar@{-}[r]&\circ\ar@{-}[r]&\bullet
}\eqno(1)
$$

$$
\xymatrix@R=0.8em{
&&&&&\bullet\\
\circ\ar@{-}[r]\ar@{-}[r]&\bullet\ar@{-}[r]&\bullet\ar@{-}[r]&\cdots\ar@{-}[r]&\bullet\ar@{-}[r]&\bullet\ar@{-}[r]\ar@{-}[u]&\bullet
}\eqno(2)
$$

$$
\xymatrix@R=0.8em{
\\
\circ\ar@{-}[r]\ar@{-}[r]&\bullet\ar@{-}[r]&\bullet\ar@{-}[r]&\cdots\ar@{-}[r]&\bullet\ar@{-}[r]&\bullet\ar@{-}[r]&\circ
}\eqno(3)
$$ where $\circ$ denotes a $(-1)$-curve, $\bullet$ denotes a $(-2)$-curve.
\end{enumerate}
\end{lemma}

\begin{lemma}[{\cite[Lemma 2.7]{Bel2}}]
\label{Pr1}
Assume that there exists a conic bundle $\phi\colon\bar{X}\rightarrow\pp^1$ such that $C$ lies in singular fiber that contains only $(-2)$-curves and one $(-1)$-curve $C$. Then every singular fiber of $\phi$ contains only $(-1)$- and $(-2)$-curves. Moreover, every $(-1)$-curve in singular fiber is minimal.
\end{lemma}

\begin{proof}
Let $F=2C+\Delta$ be the singular fiber of $\phi$ that contain $C$. Note that $\Supp(\Delta)\subset\Supp(D)$. Let $F'$ be a singular fiber of $\phi$. Then $F'=\sum n_i E_i+\Delta'$, where $E_i$ are $(-1)$-curves, $\Supp(\Delta')\subset\Supp(D)$. Since $F\sim F'$, we see that \[2a=-(2C+\Delta)\cdot(K_{\bar{X}}+D^{\sharp})=-(\sum n_i E_i+\Delta')\cdot(K_{\bar{X}}+D^{\sharp})=\sum n_i e_i\geq(\sum n_i)a,\] where $a=-C\cdot (K_{\bar{X}}+D^{\sharp})\leq -E_i\cdot (K_{\bar{X}}+D^{\sharp})=e_i$. Then $\sum n_i=2$. Hence, every singular fiber of $\phi$ is of type as in lemma \ref{Zhan0}. Moreover, every $(-1)$-curve in a singular fiber of $\phi$ is minimal.
\end{proof}

\begin{lemma}[{\cite[Lemma 2.1]{Zh}}]
\label{Zhan1}
Assume that $|C+D+K_{\bar{X}}|\neq\emptyset$. Then there exists a unique decomposition $D=D'+D''$ such that $C+D''+K_{\bar{X}}\sim 0$ and $C\cdot D_i=D''\cdot D_i=K_{\bar{X}}\cdot D_i=0$ for every irreducible component $D_i$ of $D'$
\end{lemma}

\begin{corollary}
\label{Sl1}
$|C+D+K_{\bar{X}}|=\emptyset$.
\end{corollary}

\begin{proof}
Since $C+D''+K_{\bar{X}}\sim 0$, we see that $C+D''$ is connected. Moreover, since $C+D''$ consists of rational curves, we see that $C+D''$ is a wheel of curves. Hence $D''$ is connected component of $D$. Since $K_{\bar{X}}\cdot D_i=0$ for every irreducible component $D_i$ of $D'$, we see that $D'$ correspond to Du Val singularities. So, there exists a unique non-Du Val singular point $P\in\tilde{X}$. Hence $P\in X$, a contradiction.
\end{proof}

\begin{lemma}[{\cite[Lemma 2.2]{Zh}}]
\label{Zhan2}
Assume that $|C+D+K_{\bar{X}}|=\emptyset$. Then we may assume that $C$ is a $(-1)$-curve. Assume that $\tilde{C}$ meets irreducible component $D_1,D_2,\ldots D_m$ of $D$. Then $C\cdot D_i=1$ for $i=1,2,\ldots,m$ and $D_i$, $D_j$ lie in various connected components of $D$ for $i\neq j$.
\end{lemma}

\begin{lemma}
\label{In1}
There no exists a $\tau$-invariant $(-1)$-curve $E$.
\end{lemma}

\begin{proof}
Assume that there exists a $\tau$-invariant $(-1)$-curve $E$. Let $f\colon\bar{X}\rightarrow Y$ be the contract of $E$. Then $Y$ has a point over $\rr$. A contradiction.
\end{proof}

\begin{corollary}
\label{In2}
The curve $C$ is not $\tau$-invariant curve.
\end{corollary}

\section{Proof of main Theorem}

Let $X$ be a real del Pezzo surface without points. Put $\tilde{X}$ is the same surface over complex numbers field $\cc$, $\tau$ is the Galois involution. Assume that $\rho(\tilde{X})=1$.
Put $P_1, P_2,\ldots, P_n$ are singular points of $\tilde{X}$. By {\cite{Bel1}} we see that $n\leq 4$. Since there no exists $\tau$-invariant points, we see that $n=2$ or $n=4$.
Assume that $n=4$ and $P_2=\tau(P_1)$, $P_4=\tau(P_3)$. Then $P_1, P_2$ and $P_3, P_4$ has the same type of singularities. On the other hand, by {\cite{Bel2}} we see that this is impossible. So, we may assume that $\tilde{X}$ has two singular points $P_1, P_2$ and both of them are non-Du Val (see \ref{DuVal}).

By \ref{In2} we see that $C$ is not $\tau$-invariant. Put $C'=\tau(C)$.

\begin{lemma}
\label{Int1}
The curve $C$ does not meet $C'$.
\end{lemma}

\begin{proof}
Assume that $C$ meets $C'$. Then $C\cdot C'=2m$. Assume that $m=1$. By Riemann--Roch theorem $|C+C'+K_{\bar{X}}|\neq\emptyset$. Let $\sum n_i E_i\in |C+C'+K_{\bar{X}}|$. Put $$C\equiv C'\equiv-a(K_{\bar{X}}+D^{\sharp})(\mod D),$$ $$E_i\equiv -e_i(K_{\bar{X}}+D^{\sharp})(\mod D).$$ Note that $e_i\geq a>0$ if $E_i$ is not a component of $D$ and $e_i=0$ if $E_i$ is a component of $D$. Then $$1-2a=-\sum e_i n_i\leq a(-\sum n_i).$$ So, $1\leq (2-\sum n_i)a$. Hence, $\sum n_i=1$. So, every divisor in $|C+C'+K_{\bar{X}}|$ is $\Gamma+\Delta$, where $\Gamma$ is a curve that not a component of $D$ and $\Supp(\Delta)\subset \Supp(D)$. Since $\Gamma\cdot C=\Gamma\cdot C'=0$, we see that \[\Gamma^2\leq\Gamma\cdot(\Gamma+\Delta)=\Gamma\cdot(C+C'+K_{\bar{X}})=\Gamma\cdot K_{\bar{X}}\leq\Gamma\cdot(K_{\bar{X}}+D^{\sharp})<0.\] Hence, $\Gamma$ is a $(-1)$-curve. Moreover, $\tau(\Gamma)=\Gamma$. A contradiction with lemma \ref{In1}. So,  $C+C'+K_{\bar{X}}\sim \Delta$, where $\Supp(\Delta)\subset \Supp(D)$. Note that \[0\leq(C+C'+K_{\bar{X}})\cdot\Delta=\Delta^2\leq 0.\] Then $C+C'+K_{\bar{X}}\sim 0$. On the other hand, $C$ meets a component $D_1$ of $D$. So, $0<(C+C'+K_{\bar{X}})\cdot D_1=0$. A contradiction.

Assume that $m\geq 2$. Since $\pi(C)\sim\pi(C')$, we see that $\pi(C)^2\geq 4$. Assume that $\pi(C)$ passes through only one singular point $P_1$ and $\pi(C')$ passes through singular point $P_2$.
Let $\pi^*(\pi(C))=C+\sum\beta_i D_i$. Assume that $C$ meets $D_1$. Since $\pi(C)^2\geq 4$ and $C$ is a $(-1)$-curve, we see that $\beta_1\geq 5$. Assume that $D_k$ is a $(-n)$-curve ($n\geq 3$) in $\pi^*(\pi(C))$. Since, $$0<\pi(C)\cdot(-K_{\tilde{X}})=(C+\sum\beta_i D_i)\cdot(-K_{\bar{X}}+D^\sharp),$$ we see that $\beta_k\leq\frac{1}{n-2}$, where $\beta_k$ is the coefficient before $D_k$ in $\pi^*(\pi(C))$. Hence, $D_1$ is a $(-2)$-curve.
Assume that $D_1$ meets a $(-m)$-curve ($m\geq 3$) $D_2$. Then the coefficient before $D_2$ in $\pi^*(\pi(C))$ is at least $\frac{\beta_1}{3}>1$, a contradiction. So, $D_1$ meets only $(-2)$-curves in $D$. Assume that $D_1$ meets two components $D_2, D_3$ of $D$. We see that $D_2, D_3$ are $(-2)$-curves. Let $\phi\colon\bar{X}\rightarrow\pp^1$ be a conic bundle defined by $|2C+2D_1+D_2+D_3|$. By Lemma \ref{Pr1} we see that every $(-m)$-curve ($m\geq 3$) is not in fiber $\phi$. Then the singular point $P_2$ is Du Val, a contradiction. Then $D_1$ meets only one component $D_2$ of $D$. Put $\beta_2$ is the coefficient before $D_2$ in $\pi^*(\pi(C))$. Note that $\beta_2>\beta_1\geq 5$. As above, $D_2$ does not meet $(-n)$-curves ($n\geq 3$). We obtain a rod of $(-2)$-curves $D_1, D_2,\ldots, D_k$ and $D_k$ meets two components $D_{k+1}, D_{k+2}$ of $D$. Put $\beta_1,\beta_2,\ldots,\beta_k$ are the coefficients before $D_1, D_2,\ldots, D_k$ in $\pi^*(\pi(C))$. Note that $\beta_k>\beta_{k-1}>\cdots\beta_2>\beta_1\geq 5$. Then $D_{k+1}, D_{k+2}$ are $(-2)$-curves. Let $\phi\colon\bar{X}\rightarrow\pp^1$ be a conic bundle defined by $|2C+2D_1+2D_2+\cdots+2D_k+D_{k+1}+D_{k+2}|$. By Lemma \ref{Pr1} we see that every $(-m)$-curve ($m\geq 3$) is not in fiber $\phi$. Then the singular point $P_2$ is Du Val, a contradiction.

So, we may assume that $C$ meets two components $D_1$ and $D_2$ of $D$. Then $C'$ meets also two components $D'_1$ and $D'_2$ of $D$. Put $\pi^*(\pi(C))=C+\sum\beta_i D_i$. Assume that $D_1=D'_1$ and $D_2=D'_2$. Note that $\pi(C)^2=\pi(C)\cdot\pi(C')=2m+\beta_1+\beta_2$. On the other hand, $\pi(C)^2=-1+\beta_1+\beta_2$, a contradiction. So, $D_1\neq D'_1$ and $D_2\neq D'_2$.

Assume that $D_1^2=D^2_2=-2$. Let $\phi\colon\bar{X}\rightarrow\pp^1$ be a conic bundle defined by $|2C+D_1+D_2|$. By Lemma \ref{Pr1} we see that every $(-n)$-curve ($n\geq 3$) meets $D_1$ or $D_2$. Since $P_1$ and $P_2$ have the same type of non-Du Val singularities, we see that there exist $(-n)$-curves ($n\geq 3$) $D_3$ and $D_4$ such that $D_3\cdot D_1=1$ and $D_4\cdot D_2=1$. Let $\beta_3,\beta_4$ be the coefficients before $D_3$ and $D_4$ in $\pi^*(\pi(C))=C+\sum\beta_i D_i$. Since $$0<\pi(C)\cdot(-K_{\tilde{X}})=(C+\sum\beta_i D_i)\cdot(-K_{\bar{X}}+D^\sharp),$$ we see that $\beta_3+\beta_4\leq\frac{1}{n-2}$. Since $0=\pi^*(\pi(C))\cdot D_3=\pi^*(\pi(C))\cdot D_4$, we see that $-n\beta_3+\beta_1+a=0$ and $-n\beta_4+\beta_2+a=0$, where $a\geq 0$. Then $\beta_1+\beta_2\leq\frac{n}{n-2}$. Hence, $\pi(C)^2=-1+\beta_1+\beta_2\leq 3$, a contradiction.

Assume that $D_1^2=-2$, $D_2^2=-n\leq-3$. As above, $\beta_2<\frac{1}{n-2}$ and $D_1$ does not meet $(-m)$-curves, where $m\geq 3$. Assume that $D_1$ meets two components $D_3$ and $D_4$. Then $D_3^2=D_4^2=-2$. Let $\phi\bar{X}\rightarrow\pp^1$ be a conic bundle defined by $|2C+2D_1+D_3+D_4|$. By Lemma \ref{Pr1} we see that every $(-m)$-curve ($m\geq 3$) is not in fiber $\phi$. Since $P_1$ and $P_2$ are have the same type of non-Du Val singularities, we see that $D_3$ or $D_4$ meets a $(-n)$-curve $D_5$. We may assume that $D_5$ meets $D_3$ Let $\beta_3,\beta_5$ be the coefficient before $D_3,D_5$ in $\pi^*(\pi(C))$ correspondingly. We have $\beta_3\leq 3\beta_5$. Then $\beta_1\leq5\beta_5$. As above, $\beta_5+\beta_2<\frac{1}{n-2}$. Hence, $$\pi(C)^2=-1+\beta_1+\beta_2\leq -1+5\beta_5+\beta_2< -1+4\beta_5+\frac{1}{n-2}<-1+\frac{5}{n-2}\leq 4,$$ a contradiction. So, we may assume that $D_1$ meets only one component $D_3$ of $D$. Let $\beta_3$ be the coefficient of $D_3$ in $\pi^*(\pi(C))$. Note that $$0=(C+\sum\beta_i D_i)\cdot D_1=1-2\beta_1+\beta_3.$$ Since $\beta_1>4$, we see that the coefficient $\beta_3\geq 2\beta_1-1>7$. Then $D_3$ is a $(-2)$-curve. Assume that $D_3$ meets only one component $D_4$ of $D$. Let $\beta_4$ be the coefficient of $D_4$ in $\pi^*(\pi(C))$. Since $\beta_3>7$, we see that the coefficient $\beta_4>\beta_3>4$. Then $D_4$ is a $(-2)$-curve. And so on. We obtain a rod of $(-2)$-curves. In the end there exists a $(-2)$-curves $D_i$ that meets three curves $D_{i-1}$, $D_{i+1}$ and $D_{i+2}$, where $D_{i-1}$ is a $(-2)$-curve. Assume that $D_{i+1}$ is a $(-n)$-curve ($n\geq 3$). Then $\beta_{i+1}\geq\frac{\beta_i}{n}>\frac{7}{n}$. On the other hand, $$0<\pi(C)\cdot(-K_{\tilde{X}})=(C+\sum\beta_i D_i)\cdot(-K_{\bar{X}}+D^\sharp)<1-(n-2)\frac{7}{n}<0,$$ a contradiction. So, $D_{i+1}$ and $D_{i+2}$ are $(-2)$-curves. Let $\phi\bar{X}\rightarrow\pp^1$ be a conic bundle defined by $$|2C+2D_1+2D_3+2D_4+\cdots+2D_i+D_{i+1}+D_{i+2}|.$$ By Lemma \ref{Pr1} we see that every $(-m)$-curve ($m\geq 3$) is not in fiber $\phi$. Then there exists a $(-n)$-curve ($n\geq 3$) $D_{i+3}$ that meets $D_{i+1}$ or $D_{i+2}$. We may assume that $D_{i+3}$ meets $D_{i+1}$. We have $$0=(C+\sum\beta_i D_i)\cdot D_{i+1}=\beta_i-2\beta_{i+1}+\beta_{i+3}.$$ Then $\beta_{i+1}>\frac{\beta_i}{2}>\frac{7}{2}$. Note that $-n\beta_{i+3}+\beta_{i+1}<0$. Hence, $\beta_{i+3}>\frac{7}{2n}$. On the other hand, $$0<\pi(C)\cdot(-K_{\tilde{X}})=(C+\sum\beta_i D_i)\cdot(-K_{\bar{X}}+D^\sharp)<1-(n-2)\frac{7}{2n}<0,$$ a contradiction.
\end{proof}

\begin{corollary}
\label{Int2}
The curves $\pi(C)$ and $\pi(C')$ passes through two singular points $P_1$ and $P_2$ of $\bar{X}$.
\end{corollary}

\begin{lemma}
Let $D^{(1)}$ and $D^{(2)}$ be the exception divisor of $\pi$ over $P_1$ and $P_2$ correspondingly. Then $C+C'+D^{(1)}+D^{(2)}+K_{\tilde{X}}\sim 0$.
\end{lemma}

\begin{proof}
Note that there exists a divisor $D'=\sum D_{i_j}$, where $D_{i_j}$ is a component of $D$, such that $\tilde{C}+\tilde{C}'+D'$ is a wheel. Then by the Riemann-Roch theorem $|C+C'+D'+K_{\tilde{X}}|\neq\emptyset$.
Assume that \[C+C'+D'+K_{\tilde{X}}\sim\sum n_iE_i.\] Put \[C\equiv C'\equiv-a(K_{\tilde{X}}+D^\sharp)\quad (mod D),\quad E_i\equiv-e_i(K_{\tilde{X}}+D^\sharp)\quad (mod D).\] Assume that there exists $E_i$ such that $E_i$ is not component of $D$. Then \[(C+C'+D'+K_{\tilde{X}})\cdot(K_{\tilde{X}}+D^\sharp)=(-2a+1)(K_{\tilde{X}}+D^\sharp)^2=(\sum n_iE_i)\cdot(K_{\tilde{X}}+D^\sharp)=(-\sum n_ie_i)(K_{\tilde{X}}+D^\sharp)^2.\] So, \[1-2a=-\sum n_ie_i=-\sum\limits_{e_i>0} n_ie_i\leq -a(\sum\limits_{e_i>0} n_i).\] Then $1\leq(2-\sum\limits_{e_i>0} n_i)a$. Hence, $\sum\limits_{e_i>0} n_i\leq 1$.

Assume that $\sum\limits_{e_i>0} n_i=1$. Then $C+C'+D'+K_{\tilde{X}}\sim\Gamma+\Delta$, where $\Supp(\Delta)\subset\Supp(D)$ and $\Gamma$ is an irreducible curve. As above, $\Gamma$ is a $(-1)$-curve and $\tau(\Gamma)=\Gamma$. A contradiction. So, $C+C'+D'+K_{\tilde{X}}\sim\Delta$. As above, $\Delta=0$.
\end{proof}

\begin{corollary}
\label{Wheel1}
We see that $C+C'+D^{(1)}+D^{(2)}$ is a wheel.
\end{corollary}

Put $D_1$ and $D_2$ are irreducible components of $D$ that meet $C$. Note that $D_1^2=-2$ or $D_2^2=-2$. Indeed, assume that $D_1^2<-2$ and $D_2^2<-2$. Let $g\colon\tilde{X}\rightarrow Y$ be the contraction of $C$. We see that the components of $g(D)$ generate Picard group and $g(D)$ is negative definition, a contradiction with Hodge index theorem \ref{Hodge}. So, if we consider a consequence of contraction of $(-1)$-curves in $C+C'+D^{(1)}+D^{(2)}$, we see that every $(-1)$-curve meets at least one $(-2)$-curve.
Consider contraction of pair of $(-1)$-curves in $C+C'+D^{(1)}+D^{(2)}$, start with $\tilde{C}$ and $\tilde{C}'$. We obtain a surface $Y$ such that $Y$ has no points over $\rr$. Note that the image of $C+C'+D^{(1)}+D^{(2)}$ is one of the followings
$$
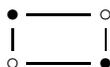
\xymatrix@R=0.8em{
\\
\bullet\ar@{-}[r]\ar@{-}[d]&\circ\ar@{-}[d]\\
\circ\ar@{-}[r]&\bullet\\
}\eqno(a)
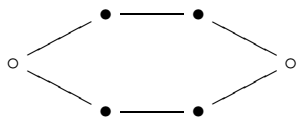
$$

$$
\xymatrix@R=0.8em{
&\bullet\ar@{-}[r]&\bullet\ar@{-}[rd]&\\
\circ\ar@{-}[ru]\ar@{-}[rd]&&&\circ\ar@{-}[ld]\\
&\bullet\ar@{-}[r]&\bullet &
}\eqno(b)
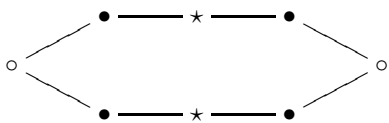
$$

$$
\xymatrix@R=0.8em{
&\bullet\ar@{-}[r]&\star\ar@{-}[r]&\bullet\ar@{-}[rd]&\\
\circ\ar@{-}[ru]\ar@{-}[rd]&&&&\circ\ar@{-}[ld]\\
&\bullet\ar@{-}[r]&\star\ar@{-}[r]&\bullet &
}\eqno(c)$$
where $\circ$ denotes a $(-1)$-curve, $\bullet$ denotes a $(-2)$-curve, $\star$ denotes $(-n)$-curve ($n\geq 2$). In first and second cases we can contract all $(-2)$-curves. We obtain a del Pezzo surface $Y_0$ with Du Val singularities and $\rho(Y_0)=1$. Moreover, $Y_0$ has one of the followings collection of singularities: $2A_1,4A_1,2A_2,2A_2+2A_1$. It is impossible (see \cite{Fur}, \cite{Ma}).

So, we have the case (c). Let $g\colon Y\rightarrow Y_0$ be the contraction of all $(-n)$-curves ($n\geq 2$). Then $Y_0$ is the del Pezzo surface with log terminal singularities and $\rho(Y_0)=1$.
Put $E_1$ and $E_2$ are $(-1)$-curve, $D_1,D_2,D_3, D_4$ are $(-2)$-curves, $D_5, D_6$ are $(-n)$-curves in (3). We assume that $E_1\cdot D_1= E_1\cdot D_2=1$, $E_2\cdot D_3= E_2\cdot D_4=1$. Note that the linear system $|2E_1+D_1+D_2|$ induces $\pp^1$-fibration $\phi\colon Y\rightarrow\pp^1$. We see that $D_5$ and $D_6$ are section of $\phi$. By lemma \ref{Zhan0} we see that there exist exactly one singular fiber of type (3), i.e. there exists a unique singular fiber $F$ that consists of two $(-1)$-curves $E_3$ and $E_4$. Then $Y_0$ has $\tau$-invariant point $Q=E_3\cap E_4$, a contradiction.

\end{document}